\documentclass[11pt]{amsproc}

\usepackage[top=1in,bottom=1in,left=1.5in,right=1.5in]{geometry}
\usepackage{amsfonts} %\mathbb{R}, etc.
\usepackage{amsmath}
\usepackage{mathtools}          %loads amsmath as well
\DeclarePairedDelimiter\Floor\lfloor\rfloor

\usepackage{amssymb}
\usepackage{amsthm}
\usepackage{setspace}
\usepackage{mathrsfs}
\usepackage{subfigure}
\usepackage{url}
\usepackage{booktabs}
\usepackage{ifthen}
\usepackage{tikz}
\usetikzlibrary{calc}
\usepackage{enumerate}

\newcommand{\floor}[1]{\left\lfloor#1\right\rfloor}
\newcommand{\ceil}[1]{\left\lceil#1\right\rceil}
\newcommand{\discrepancy}[1]{\text{dis}(#1)}

\newtheorem{thm}{Theorem}[section]
\newtheorem{lem}[thm]{Lemma}

\newtheorem{case}{Case}

\newtheorem{conj}[thm]{Conjecture}

\newtheorem{openquestion}{Open question}

\numberwithin{equation}{section}

\newcommand \track[1] {{\color{blue} #1}}
\newcommand \sout[1] {{\color{red} #1}}
\renewcommand\track[1]{{\color{black} #1}}
\renewcommand\sout[1]{{\iffalse #1 \fi}}

\begin{document}

\title{The Cordiality Game and the Game Cordiality Number}

\author{Elliot Krop}
\email{Elliot Krop (\tt elliotkrop@clayton.edu)}

\author{Aryan Mittal}
\email{Aryan Mittal (\tt aryan.mittal@gatech.edu)}

\author{Michael C. Wigal}
\email{Michael Wigal (\tt \sout{wigal@gatech.edu} \track{wigal@illinois.edu})}

\address{Department of Mathematics, Clayton State University}
\address{Georgia Institute of Technology}
\address{Department of Mathematics, University of Illinois Urbana-Champaign}

\date{\today}

\maketitle

\begin {abstract}
The \emph{cordiality game} is played on a graph $G$ by two players, Admirable (A) and Impish (I), who take turns selecting \track{unlabeled} vertices of $G$. Admirable labels the selected vertices by $0$ and Impish by $1$, and the resulting label on any edge is the sum modulo $2$ of the labels of the vertices incident to that edge. The two players have opposite goals: Admirable attempts to minimize the number of edges with different labels as much as possible while Impish attempts to maximize this number. When both Admirable and Impish play their optimal games, we define the \emph{game cordiality number}, $c_g(G)$, as the absolute difference between the number of edges labeled zero and one. Let $P_n$ be the path on $n$ vertices. We show $c_g(P_n)\le \frac{n-3}{3}$ when $n \equiv 0 \pmod 3$, $c_g(P_n)\le \frac{n-1}{3}$ when $n \equiv 1 \pmod 3$, and $c_g(P_n)\le \frac{n+1}{3}$ when $n \equiv 2\pmod 3$. Furthermore, we show a similar bound, $c_g(T) \leq \frac{|T|}{2}$ holds for any tree $T$.
\\[\baselineskip] 2020 Mathematics Subject
      Classification:  05C57, 05C78
\\[\baselineskip]
\end {abstract}

\section{Introduction and Definitions}

It is natural to study combinatorial functions under the regime of two-person games, often presenting new perspective\track{s} to old problems. One such example is the classic cops and robbers formulation of tree-width \cite{ST}. More examples include game chromatic number, see \cite{BGKZ} and \cite{TZ}, and the more recent game domination number \cite{BKR}, which has bloomed into numerous publications, of which we only list a few, e.g., \cite{BHZ,HK,KWZ}, and the monograph \cite{BHKR}. For more examples of graph-theoretic games, see the extended bibliography \cite{F} for more references.

With this in mind, we define a new game to study a simple ``algebraic balance" in graphs, which is motivated from the following definition. A graph $G$ is \emph{cordial} if there exists a labeling of the vertices of $G$ by $0$ and $1$, where labels on edges are defined as the sum of incident vertex labels modulo $2$ such that:
\begin{itemize}
    \item the absolute difference in the number of vertices labeled $0$ and $1$ is no more than one;
    \item the absolute difference in the number of edges labeled $0$ and $1$ is no more than one.
\end{itemize} In other words, a graph is cordial, if there exists a bipartition of its vertex set, such that the size of each part is $\ceil{|V(G)|/2}$ or $\floor{|V(G)|/2}$, and the number of edges that goes across the cut induced by the bipartition is either $\ceil{|E(G)|/2}$ or $\floor{|E(G)|/2}$. Cordial labelings were first defined in \cite{Cahit} and extended to $k$-cordial labelings in \cite{Hovey}, which includes several difficult and outstanding problems. One particular simple and natural problem is the $3$-cordial conjecture.

\medskip

A graph $G$ is $3$-cordial, if there exists a vertex labeling of $G$ satisfying the following two conditions: 
\begin{itemize}
\item The vertices are labeled by the integers $0, 1,$ and $2$ such that for any integer in $\{0,1,2\}$, there is at most one more vertex labeled by one of the integers than of vertices labeled by any other fixed integer. 
\item If for any edge, the label on that edge is found by summing its incident vertices modulo $3$, then there is at most one more edge labeled by one of the integers than of edges labeled by a different fixed integer.
\end{itemize}

\begin{conj}[Hovey \cite{Hovey}]
All graphs are $3$-cordial.
\end{conj}

We know of no progress made on this problem in the last thirty years. Some examples of other extensions and variants of cordiality problems from \cite{Hovey} can be seen in \cite{Cichacz, CGT, DKN, PP, PW}.
\medskip

Let $G = (V,E)$ be a graph. The \emph{cordiality game} is played on $G$ by two players, Admirable and Impish, who take turns selecting the \track{unlabeled} vertices of $G$. Admirable labels selected vertices by $0$ and Impish labels selected vertices by $1$. The labels on edges are then determined by the sum of incident vertex labels modulo $2$. Admirable's goal is to produce a labeling of $E$ with the minimum number of different labels while Impish attempts to produce the labels of $E$ with the maximum number of different labels. In other words, after all the vertices are labeled, if we let $e_0$ be the number of edges labeled by $0$ and $e_1$ be the number of edges labeled by $1$, then we define the \emph{discrepancy} to be $d = |e_1 - e_0|$. Then Admirable attempts to minimize $d$ and Impish attempts to maximize $d$. Note at the end of the game, the vertex labelings induce a balanced bipartition of the graph. 

\medskip

We define the \emph{game cordiality number}, $c_g(G)$, to be the value of $d$ when both players play optimally. Further, to prove our claimed bounds, we create a variant of the cordiality game where Impish starts rather than Admirable. In this case, the discrepancy, assuming optimal play of both competitors, is denoted $c_g'(G)$, which we call the Impish-starts game cordiality number. Finally, consider a further variant of the cordiality game, in which Impish starts and may pass one time during play, so that Admirable moves twice consecutively at some point in the game. For this game, the discrepancy, assuming optimal play of both parties, is denoted $c_g^*(G)$. Note that $c_g'(G) \le c_g^*(G)$, as Impish may always choose a strategy that avoids passing \sout{their} \track{his} turn.

\medskip

Our cordiality game can be interpreted as a classic positional game under a small adjustment. First introduced by Chv{\'a}tal and Erd{\H o}s \cite{CE}, the Maker-Breaker game has received significant attention prior in the literature, e.g., see \cite{HKSS,Krivelevich}. A family ${\mathcal F}$ of winning subsets of some set $X$ is known to both players prior. The players alternate between selecting \track{ previously unselected} elements. At the end, Maker wins if he obtains a set of ${\mathcal F}$, otherwise Breaker wins. For some fixed integer $k \ge 0$, let ${\mathcal F}_k$ be the family of vertex sets which are parts of a balanced bipartition of discrepancy at most $k$. Note then for the cordiality game, Maker would play the role of Admirable, Breaker taking the role of Impish. Clearly, setting $k = c_g(G)$, a winning strategy for one game translates to the other.

In Section $2$, we present an upper bound for $c_g(P_n)$ when $P_n$ is a path on $n$ vertices. We \sout{show} \track{prove}

\[
  c_g(P_n) \le
  \begin{cases}
                                   \frac{n-3}{3} & n \equiv 0 \pmod 3\\
                                   \frac{n-1}{3} & n \equiv 1 \pmod 3\\
                                   \frac{n+1}{3} & n \equiv 2 \pmod 3\\
  \end{cases}
\]

In Section $3$, we show that for any tree $T$, $c_g(T)\le \frac{1}{2}n$. In Section $4$, we introduce a variant of the cordiality game and end with some open questions.

We finish this section with some notation. For a graph $G$, we let $\phi$ denote the \emph{cut function} for $G$. That is, for all $S \subseteq V(G)$, $\phi(S) = |\{uv \in E(G) : u \in S, v \not \in S\}|$. We let \sout{$|\discrepancy{S}|$} \track{$\discrepancy{S}$} denote the discrepancy of the cut induced by $S$, where
$\discrepancy{S} =  \phi(S) - (|E(G)| - \phi(S)) = 2\phi(S) - |E(G)|$. Note then the objective of the cordiality game is \sout{$|\discrepancy{S}|$} \track{$\discrepancy{S}$} where $S$ is the set of vertices labeled 0 or 1. \sout{Further note by parity,  $\discrepancy{S}$ is always either odd or even.} \track{Furthermore, we note the parity of $\discrepancy{S}$ is always the same as that of $|E(G)|$.} A \emph{balanced partition} of a graph $G$ is a partition of $V(G)$ into two sets, say $S_1$ and $S_2$, such that $||S_1| - |S_2|| \le 1$.

\section{Paths}

The problem of finding the exact value of $c_g(P_n)$ is surprisingly involved. We begin with calculating the game cordiality number for small paths.

\begin{lem}\label{base}
    We have the following:

    \begin{enumerate}[i.]
        \item $c_g(P_3) = c'_g(P_3) = c_g^*(P_3) = 0$,
        \item $c_g(P_4) = c'_g(P_4) = c_g^*(P_4) = 1$,
        \item $c_g(P_5) = c'_g(P_5) = c_g^*(P_5) \le 2$,
        \item $c_g(P_6) = c'_g(P_6) = c^*_g(P_6) = 1$.
    \end{enumerate}
\end{lem}

\begin{proof}
    Let $P_n = v_1 v_2 \cdots v_n$ and let $S,\overline{S}$ be a balanced bipartition of $V(P_n)$.
    
    In the case $n = 3$, \sout{$|\discrepancy{S}|$} \track{$\discrepancy{S}$} is always even. As $|E(P_3)| = 2$, $\{v_2\},\{v_1,v_3\}$ is the only bipartition that achieves discrepancy greater than 0. Thus if Admirable plays $v_2$ first, this guarantees a bipartition of discrepancy zero. In the case Impish plays first, Admirable just needs to avoid playing $v_2$. Thus we may conclude $c_g(P_3) = c'_g(P_3) = c_g^*(P_3) =0$.

    When $n = 4$, \sout{$|\discrepancy{S}|$} \track{$\discrepancy{S}$} is always odd. By observation, $\{v_1,v_3\},\{v_2,v_4\}$ is the only balanced bipartition that achieves discrepancy greater than 1. Clearly if Admirable has a strategy of playing once from each set, regardless of who starts, this guarantees a resulting bipartition of discrepancy 1. Note as \sout{$|\discrepancy{S}| \ge 0$} \track{$\discrepancy{S} \ge 0$} and odd, we may conclude $c_g(P_4) = c_g'(P_4) = c_g^*(P_4) = 1$.

    In the case $n = 5$, \sout{$|\discrepancy{S}|$} \track{$\discrepancy{S}$} is always even. As $|E(P_5)| = 4$, the only balanced bipartition that achieves discrepancy greater than 2 is $\{v_1,v_3,v_5\},\{v_2,v_4\}$. Again if Admirable takes a strategy of playing once from each set, this is a strategy guaranteeing $c_g(P_5) = c'_g(P_5) = c^*_g(P_5) \le 2$.

    We now work towards understanding our game variants on $P_6$. By parity, \sout{$|\discrepancy{S}|$} \track{$\discrepancy{S}$} is odd and \sout{$|\discrepancy{S}| \ge 1$} \track{$\discrepancy{S} \ge 1$} for any balanced bipartition $S$,$\overline{S}$. There are only two choices for $S$ that achieve discrepancy 5, namely $\{v_1,v_3,v_5\}$ and its complement $\{v_2,v_4,v_6\}$. We now observe the balanced bipartitions which achieve discrepancy 3. To do so, we would either need exactly 1 or 4 edges crossing the cut induced by the bipartition. The former case only occurs with $S$ being $\{v_1,v_2,v_3\}$ or $\{v_4,v_5,v_6\}$. The latter case occurs only with $S$ being $\{v_2,v_3,v_5\}, \{v_2,v_4,v_5\}, \{v_1,v_4,v_6\}$, or $\{v_1,v_3,v_6\}$. Note then for Admirable to obtain discrepancy 1 in the cordiality game, there is in total 8 ``bad labelings" to avoid. \track{We list them as follows.
    \begin{align*}
         &\{v_1,v_3,v_5\}, \;\;\; \{v_2,v_4,v_6\},\:\:\; \{v_1,v_2,v_3\}, \;\;\;\; \{v_4,v_5,v_6\},\\
        &\{v_2,v_3,v_5\}, \;\;\; \{v_2,v_4,v_5\}, \;\;\; \{v_1,v_4,v_6\}, \;\;\; \{v_1,v_3,v_6\}.
    \end{align*}
    }
    We now give a strategy for Admirable which guarantees a balanced cut of discrepancy 1, i.e., $c_g(P_6) = 1$. We may assume Admirable plays on $v_1$ first. Note then there is only four bad labelings for Admirable to avoid now, $\{v_1,v_3,v_5\}$, $\{v_1,v_2,v_3\}$, $\{v_1,v_4,v_6\}$, or $\{v_1,v_3,v_6\}$. Note $v_3$ is in three of these bad sets. As Admirable plays first and $|P_6|$ is even, we may suppose Admirable never plays $v_3$. Thus to avoid the final bad set $\{v_1,v_4,v_6\}$, on Admirable's second move, we may assume Admirable plays on either $v_2$ or $v_5$. This guarantees that Admirable labels a set $S$ such that \sout{$|\discrepancy{S}| = 1$} \track{$\discrepancy{S} = 1$}.

    We now give a strategy for Admirable in which Impish plays first. In the case where Impish is also allowed to pass \sout{their}\track{his} turn to Admirable once, the analysis is almost the same which we note at the end. By symmetry, we may suppose Impish plays $v_1$,$v_2$ or $v_3$. 

    If Impish plays on $\{v_1,v_3\}$ on \sout{their}\track{his} first move, then Admirable plays $\{v_1,v_3\}$ for \sout{their}\track{his} first move as well. Note then the only bad labeling for Admirable to avoid now is $\{v_1,v_4,v_6\}$ \sout{and}\track{or} $\{v_2,v_3,v_5\}$ depending on \sout{their}\track{his} first move choice. This can clearly be achieved. If Impish plays $v_2$, then Admirable plays $v_5$. Note then the only bad labelings to avoid are $\{v_4,v_5,v_6\}$ and $\{v_1,v_3,v_5\}$. To avoid these, if Impish plays on $\{v_4,v_6\}$ on \sout{their}\track{his} second move, then Admirable plays on $\{v_4,v_6\}$. In a similar manner, if Impish plays on $\{v_1,v_3\}$ on \sout{their}\track{his} second move, then Admirable plays on $\{v_1,v_3\}$. If Impish chooses to pass on \sout{their}\track{his} second turn, Admirable simply plays on $\{v_1,v_3\}$ on \sout{their}\track{his} second move and then $\{v_4,v_6\}$ on \sout{their}\track{his} third move. Thus we may conclude $c_g'(P_6)= c_g^*(P_6) = 1$. 

\end{proof}

We now prove an upper bound for all paths.

\begin{thm}\label{ub}
For any integer $n\ge 3$,
\[
  c_g(P_n) \le
  \begin{cases}
                                   \frac{n-3}{3} & n \equiv 0 \pmod 3\\
                                   \frac{n-1}{3} & n \equiv 1 \pmod 3\\
                                   \frac{n+1}{3} & n \equiv 2 \pmod 3\\
  \end{cases}
\]
\end{thm}

\begin{proof}

Let $P_n = v_1 v_2 \cdots v_n$. We now proceed by induction on $n$.  By Lemma \ref{base}, we may suppose that the theorem holds for all paths of order less than $n>6$ and consider $P_n$. We describe a strategy for Admirable which, regardless of the moves Impish makes, will produce the claimed upper bounds.

Let $P$ be the subpath of $P_n$ induced on $v_1,\dots, v_{n-6}$ and $P'$ be the $P_6$ subpath of $P$ induced by $v_{n-5}, \dots, v_{n}$. We will prove the inequality
\[
  c_g(P_n) \le
  \begin{cases}
                                   2\Floor{\frac{n}{6}}-1 & n \equiv 0 \pmod 6\\
                                   2\Floor{\frac{n}{6}} & n \equiv 1, 3 \pmod 6\\
                                   2\Floor{\frac{n}{6}}+1 & n \equiv 2, 4 \pmod 6\\
                                   2\Floor{\frac{n}{6}}+2 & n \equiv 5 \pmod 6\\
  \end{cases}
\]
which can be easily reduced to the claimed statement.

\medskip

Suppose $n$ is odd. Admirable begins by playing an optimal game on $P$. From that point, Admirable follows Impish between $P$ and $P'$. If Impish plays on $P$, Admirable continues playing an optimal game on $P$. If Impish plays on $P'$, then Admirable follows with a next optimal move on $P'$. This strategy results in an Impish-starts game on $P'$ and an Admirable-starts game on $P$. Recall by Lemma \ref{base}, $c'_g(P') = c'_g(P_6) = 1$. We apply the induction hypothesis on $P$ and consider that the edge between $P$ and $P'$, $v_{n-6}v_{n-5}$, as an unknown label. Since $n-6 \pmod 6 = n \pmod 6$, the values in the proposed upper bound of $c_g(P_n)$ follow\track{s} by adding $2$ to the bounds for $c_g(P_{n-6})$.

\medskip

Next, we suppose that $n$ is even. The strategy for Admirable is the same as the previous case. Admirable begins by playing an optimal game on $P$. From that point, Admirable follows Impish between $P$ and $P'$. If Impish plays on $P$, Admirable continues playing an optimal game on $P$. If Impish plays on $P'$, Admirable follows with a next optimal move on $P'$. Due to the parity of $n$ in this case, this strategy allows for several variants of the cordiality game. If Impish completes a game on $P$ before moving to $P'$, then the game on $P'$ becomes an Admirable-starts game. If Impish plays $P'$ before completing the game on $P$, then the game on $P'$ becomes an Impish-starts game. If Impish does not finish the game on $P'$, returns to $P$, and completes the game there, then Admirable must make the next move on $P'$, which is equivalent to playing an Impish-starts game on $P'$ where Impish may pass once. Regardless of Impish strategy, by Lemma \ref{base}, $c_g(P_6) = c'_g(P_6) = c^*_g(P_6) = 1$. Thus by induction on $P$, and treating $v_{n-6}v_{n-5}$ as an unknown label, this produces the claimed upper bounds for $c_g(P_n)$ as before.
\end{proof}

For a lower bound, our methods are ineffective, and so from some empirical considerations ($n\le 12$), we ask the following speculative question.

\begin{openquestion}\label{lbc}
Does $c_g(P_n)$ assume the values $0$ and $1$  for infinitely many integers $n$?  
\end{openquestion}

\section{Trees}

 \sout{ Our main difficulty in extending the upper bound present in Theorem \ref{ub} to general trees is the following. If we follow the method used for paths by separating a small branch from the tree to apply induction, that branch may be of odd order.} \track{In Theorem \ref{ub}, if our input path $P_n$ is long, i.e. $n \ge 6$, we break the instance into two smaller paths, $P_{n-6}$ and $P_6$, and then induct with a strategy of Admirable following Impish. There are challenges to extending this proof strategy to an arbitrary tree $T$. A subtree $B$ of a tree $T$ is a \emph{branch}, if there exists a subtree $T'$ of $T$ such that $B\cup T'=T$ and $|V(T')\cap V(B)|=1$. Ideally, we would like to find a small branch of constant size to break the instance into two smaller trees. If we do so naively,  we no longer have control on what this branch looks like, in particular, the branch may be of odd order.} Using the strategy of Admirable following Impish, Impish may play the odd ordered branch entirely with some vertices still unlabeled in the remainder of the tree, which would lead to Admirable playing twice in-a-row on the remainder of the tree. Thus, to avoid this potential difficulty in analysis, we consider how to remove a branch that is small enough to label, yet also of even order.

\begin{thm}\label{tree}
For any tree \track{$T$} of order $n$,
\[c_g(T)\le \frac{1}{2}n.\] 
\end{thm}

\begin{proof}
We proceed by induction on the order $n$. Note that the theorem holds trivially for trees of order $1,2$, and $3$, so we suppose the statement is true for all trees of order less than $n > 3$ and consider a tree $T$ of order $n$.

\medskip

We describe Admirable's strategies and calculate discrepancies for various possible branches of $T$. We will then \track{show} that our bound holds in every instance. \sout{In this process, we define $B$, a subgraph of $T$ as a \emph{branch}, if $T'$ is a subgraph of $T$ so that $B\cup T'=T$ and $|V(T')\cap V(B)|=1$.} \track{Let $B$ be a branch of $T$, let $T'$ be a subtree of $T$ such that $B\cup T'=T$ and $|V(T')\cap V(B)|=1$, and} let $v$ be the vertex which is both in $B$ and $T'$. \track{Finally, we let $S$ be the subgraph of $T$ induced by the vertices of $B - v$.}

\track{Throughout the proof, we will assume Admirable begins playing an optimal game on $T'$, and then following Impish between $T'$ and $S$ depending on where Impish plays. In the case Impish played on $T'$, and $T'$ has no remaining unlabeled vertices, Admirable will begin playing on $S$ if possible. We also will choose $B$ such that $|S|$ is always even. We remark that as $|T'| + |S| = |T|$, we have
\begin{itemize}
    \item If $|T'|$ is even, then the game on $S$ is either an Admirable-start game or an Impish-start game where Impish may pass up to one time during play.
    \item If $|T'|$ is odd, then the game on $S$ is an Impish-start game. 
\end{itemize}
Note in the case $|T'|$ is even, $S$ is an Admirable-start game if and only if every vertex of $T'$ has been previously played.}

\medskip
\setcounter{case}{0}
\begin{case}\label{Case:1}
Branch $B$ isomorphic to $P_5$ with end-vertex $v$ in $T'$
\end{case}

If $T$ contains a branch isomorphic to $P_5=(v,v_1,v_2,v_3,v_4)$, where $\deg(v_i)=2$ for $i=1,2,3$ and $\deg(v_4)=1$, in $T$, then we define $T'=T-\{v_1,\dots, v_4\}$ and let $S$ be the subgraph of $T$ induced by $v_1,v_2,v_3,v_4$. \sout{ Admirable begins by playing an optimal game on $T'$ and then follows Impish between $T'$ and $S$, depending on where Impish plays.}By Lemma \ref{base}, \sout{ $c'_g(P_4)=c^*_g(P_4)=1$}\track{$c_g(P_4) = c'_g(P_4) = c^*_g(P_4) = 1$}, we can apply the induction hypothesis to $T'$ and consider the edge between $T'$ and $S$ as one with an unknown label, to produce the bound
\[c_g(T)\le c_g(T')+2\le \frac{1}{2}(n-4)+2=\frac{1}{2}n.\]{\hfill \rule{4pt}{7pt}} 

\begin{case}\label{Case:2}
Branch $B$ isomorphic to $P_3$ with central vertex $v$ in $T'$.
\end{case}

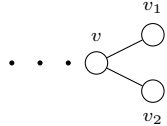
\begin{figure}[ht]
\begin{center}
\begin{tikzpicture}[scale=.75]
\tikzstyle{vert}=[circle,fill=black,inner sep=3pt]
\tikzstyle{overt}=[circle, draw, inner sep=3pt]

\node[overt, label=above:\tiny{$v$}] (v) at (1,1) {};
\node[overt, label=above:\tiny{$v_1$}] (v_1) at (2,1.5) {};
\node[overt, label=below:\tiny{$v_2$}] (v_2) at (2,.5) {};

\draw[color=black] 
(v)--(v_1) (v)--(v_2);

\filldraw[color=black]
    (.5,1) circle (1pt)
    (0,1) circle (1pt)
    (-.5,1) circle (1pt)
;

\end{tikzpicture}
%\caption{}
%\label{example}
\end{center}
\end{figure}

Let the vertices of $B$ be labeled as the above figure, and let $S$ be the subgraph induced by $\{v_1,v_2\}$. \sout{Admirable begins play on $T'$. If Impish plays on $T'$, Admirable follows Impish and plays on $T'$. If at any point, Impish plays on $S$, Admirable plays the next move on $S$ as well. Note if $|T'|$ is odd, it may occur that Admirable must play on $S$ before Impish. In this case, Impish is forced to play the last available vertex, which belongs to \sout{$P$}. }\track{As Admirable plays first on $T'$, and follows Impish between $T'$ and $S$, playing on $S$ first if and only if $T'$ is completely labeled,} \sout{In all cases,}the edges $vv_1$ and $vv_2$ \track{will always} have different labels. Thus the claim follows by induction on $T'$. {\hfill \rule{4pt}{7pt}}

\begin{case}\label{Case:3}
Branch $B$ isomorphic to $P_5=\{v_1,v_2,v,v_3,v_4\}$ with vertex $v$ in $T'$.
\end{case}
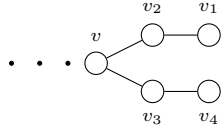
\begin{figure}[ht]
\begin{center}
\begin{tikzpicture}[scale=.75]
\tikzstyle{vert}=[circle,fill=black,inner sep=3pt]
\tikzstyle{overt}=[circle, draw, inner sep=3pt]

\node[overt, label=above:\tiny{$v$}] (v) at (1,1) {};
\node[overt, label=above:\tiny{$v_2$}] (v_2) at (2,1.5) {};
\node[overt, label=above:\tiny{$v_1$}] (v_1) at (3,1.5) {};
\node[overt, label=below:\tiny{$v_3$}] (v_3) at (2,.5) {};
\node[overt, label=below:\tiny{$v_4$}] (v_4) at (3,.5) {};

\draw[color=black] 
(v)--(v_2)--(v_1) (v)--(v_3)--(v_4);

\filldraw[color=black]
    (.5,1) circle (1pt)
    (0,1) circle (1pt)
    (-.5,1) circle (1pt)
;

\end{tikzpicture}
%\caption{}
%\label{example}
\end{center}
\end{figure}

Let \track{$B$ be labeled as the above figure and let} $S$ be the subgraph induced by $\{v_1,v_2,v_3,v_4\}$. We now give strategies for Admirable \track{on $S$} that produces a discrepancy of \track{at most} $2$ for the \track{edges induced by the} branch $B$, regardless of the value of the label on $v$. Note this occurs if both Admirable and Impish have played on $\{v_2,v_3\}$. To see this, $vv_2$ and $vv_3$ must have different signs. Thus we have discrepancy at most $2$ from edges $v_1v_2$ and $v_3v_4$.

We begin with the Admirable-start\sout{s} game\track{s}. We may suppose Admirable plays $v_2$, and then plays on $\{v_1,v_4\}$. This guarantees both players played on $\{v_2,v_3\}$.

For \track{an} Impish-starts game, \sout{I}\track{i}f Impish plays on $\{v_2,v_3\}$, then Admirable immediately follows by playing on $\{v_2,v_3\}$ as well. If Impish plays on $\{v_1,v_4\}$, then Admirable again immediately follows by playing on $\{v_1,v_4\}$ as well. Note this strategy is robust against the case where Impish is also allowed to pass \sout{their}\track{his} turn. {\hfill \rule{4pt}{7pt}}

\medskip

\begin{case}\label{Case:4}
Branch $B$ is illustrated below with vertex $v$ in $T'$.
\end{case}

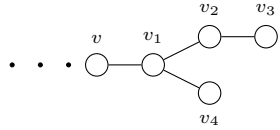
\begin{figure}[ht]
\begin{center}
\begin{tikzpicture}[scale=.75]
\tikzstyle{vert}=[circle,fill=black,inner sep=3pt]
\tikzstyle{overt}=[circle, draw, inner sep=3pt]

\node[overt, label=above:\tiny{$v$}] (v) at (1,1) {};
\node[overt, label=above:\tiny{$v_1$}] (v_1) at (2,1) {};
\node[overt, label=above:\tiny{$v_2$}] (v_2) at (3,1.5) {};
\node[overt, label=above:\tiny{$v_3$}] (v_3) at (4,1.5) {};
\node[overt, label=below:\tiny{$v_4$}] (v_4) at (3,.5) {};

\draw[color=black] 
(v)--(v_1)--(v_2)--(v_3) (v_1)--(v_4);

\filldraw[color=black]
    (.5,1) circle (1pt)
    (0,1) circle (1pt)
    (-.5,1) circle (1pt)
;

\end{tikzpicture}
%\caption{}
%\label{example}
\end{center}
\end{figure}

Let \track{the vertices of $B$ be labeled as in the figure, and let} $S$ be the subgraph induced by $\{v_1,v_2,v_3,v_4\}$. \sout{We show the move-order sequences that produces a discrepancy of $2$ for $B$, regardless of the value of the label on $v$.}\track{ We give a strategy for Admirable which produces a discrepancy of at most $2$ for the edges induced by $B$, regardless of the value of the label on $v$.}\sout{$S$ that}  Note that $v_4v_1v_2v_3$ induce\track{s} a path, thus it is sufficient to bound the discrepancy of the edges $v_4v_1, v_1v_2, v_2v_3$ by one. This is true by Lemma \ref{base}. {\hfill \rule{4pt}{7pt}}

\begin{case}\label{Case:5}
Branch $B$ isomorphic to $P_5=\{v_1,v,v_2,v_3,v_4\}$ with vertex $v$ in $T'$.
\end{case}
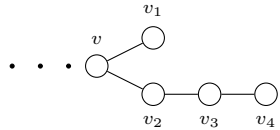
\begin{figure}[ht]
\begin{center}
\begin{tikzpicture}[scale=.75]
\tikzstyle{vert}=[circle,fill=black,inner sep=3pt]
\tikzstyle{overt}=[circle, draw, inner sep=3pt]

\node[overt, label=above:\tiny{$v$}] (v) at (1,1) {};
\node[overt, label=above:\tiny{$v_1$}] (v_1) at (2,1.5) {};
\node[overt, label=below:\tiny{$v_2$}] (v_2) at (2,.5) {};
\node[overt, label=below:\tiny{$v_3$}] (v_3) at (3,.5) {};
\node[overt, label=below:\tiny{$v_4$}] (v_4) at (4,.5) {};

\draw[color=black] 
(v)--(v_1) (v)--(v_2)--(v_3)--(v_4);

\filldraw[color=black]
    (.5,1) circle (1pt)
    (0,1) circle (1pt)
    (-.5,1) circle (1pt)
;

\end{tikzpicture}
%\caption{}
%\label{example}
\end{center}
\end{figure}

Let \track{the vertices of $B$ be labeled as the figure above, and let} $S$ be the subgraph induced by $\{v_1,v_2,v_3,v_4\}$. \sout{We show the move-order sequences}\track{As in the prior case, we give a strategy for Admirable on $S$} that produces a discrepancy of \track{at most} $2$ for $B$, regardless of the value of the label on $v$. Note this occurs if Admirable plays once on $S_1 = \{v_1,v_2\}$ and once on $S_2 = \{v_3,v_4\}$. 

For Admirable-start games, we first play $v_1$ and then Admirable on \sout{their}\track{his} second move plays on $\{v_3,v_4\}$. When Impish plays first, if Impish plays on $S_i$ for $i \in \{1,2\}$, then Admirable follows by playing on $S_i$ as well. In the case Impish is allowed to pass and does so on \sout{their}\track{his} second move, Admirable plays in a way so that he plays on both $S_1$ and $S_2$. {\hfill \rule{4pt}{7pt}}

%\pagebreak

\begin{case}\label{Case:6}
Branch $B$ isomorphic to $P_7=\{v_1,v_2,v_3,v,v_4,v_5,v_6\}$ with vertex $v$ in $T'$.
\end{case}

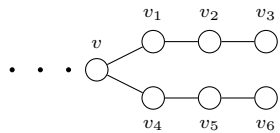
\begin{figure}[ht]
\begin{center}
\begin{tikzpicture}[scale=.75]
\tikzstyle{vert}=[circle,fill=black,inner sep=3pt]
\tikzstyle{overt}=[circle, draw, inner sep=3pt]

\node[overt, label=above:\tiny{$v$}] (v) at (1,1) {};
\node[overt, label=above:\tiny{$v_1$}] (v_1) at (2,1.5) {};
\node[overt, label=above:\tiny{$v_2$}] (v_2) at (3,1.5) {};
\node[overt, label=above:\tiny{$v_3$}] (v_3) at (4,1.5) {};
\node[overt, label=below:\tiny{$v_4$}] (v_4) at (2,.5) {};
\node[overt, label=below:\tiny{$v_5$}] (v_5) at (3,.5) {};
\node[overt, label=below:\tiny{$v_6$}] (v_6) at (4,.5) {};

\draw[color=black] 
(v)--(v_1)--(v_2)--(v_3) (v)--(v_4)--(v_5)--(v_6);

\filldraw[color=black]
    (.5,1) circle (1pt)
    (0,1) circle (1pt)
    (-.5,1) circle (1pt)
;

\end{tikzpicture}
%\caption{}
%\label{example}
\end{center}
\end{figure}

Let \track{the vertices of $B$ be labeled as above and} $S$ be the subgraph induced by $\{v_1,v_2,v_3,v_4,v_5,v_6\}$. \sout{We show the move-order sequences}\track{Again, we give a strategy for Admirable on $S$} that produces a discrepancy of \track{at most} $2$ for the branch $B$, regardless of label on $v$. \sout{To achieve this, there are two strategies available for Admirable. In Strategy 1, Admirable will always guarantee both Admirable and Impish play on $\{v_1,v_4\}$. Conditioned on this constraint, to guarantee our desired discrepancy, it is enough for Admirable to avoid labeling the following ``bad" sets $\{v_1,v_3,v_5\}$, $\{v_2,v_4,v_6\}$, $\{v_1,v_2,v_3\}$, and $\{v_4,v_5,v_6\}$ as these are the sets inducing cuts of discrepancy $4$. In Strategy 2, either only Admirable or Impish played on $\{v_1,v_4\}$. Note under this assumption, Admirable must guarantee the edges induced by $S$ have discrepancy zero. For this, it is enough to note $\{v_1,v_2,v_4\}$, $\{v_1,v_4,v_5\}$, $\{v_2,v_3,v_6\}$, and $\{v_3,v_5,v_6\}$ are sets in which if Admirable labels, then Admirable wins under this Strategy.} \track{To achieve this, we describe the possible winning labelings for Admirable. In the first case, suppose both Admirable and Impish play on $\{v_1,v_4\}$. Conditioned on this constraint, by inspection, the sets inducing cuts of discrepancy at most $2$ on $S$ are the following,
   \begin{align*}
         &\{v_1,v_2,v_5\},\;\;\; \{v_1,v_2,v_6\},\;\;\; \{v_1,v_3,v_6\}, \;\;\; \{v_1,v_5,v_6\},\\
         &\{v_2,v_3,v_4\}, \;\;\; \{v_2,v_4,v_5\},  \;\;\; \{v_3,v_4,v_5\}, \;\;\; \{v_3,v_4,v_6\}.
    \end{align*}
As the edges $vv_1$ and $vv_4$ must have different labelings, if Admirable labels one of the above sets, this would be considered a winning labeling. Now consider the other case, either only Admirable or Impish play on $\{v_1,v_4\}$. Under this assumption, Admirable must guarantee the edges induced by $S$ have discrepancy $0$ under his labeling. By inspection, the following are such labelings, 
   \begin{align*}
         &\{v_1,v_2,v_4\}, \;\;\; \{v_1,v_4,v_5\}, \:\:\; \{v_2,v_3,v_6\}, \;\;\;\; \{v_3,v_5,v_6\}.
    \end{align*}
Thus, in all cases, we will give a strategy for Admirable for playing on $S$ that guarantees one of the above twelve labelings.}

\sout{For Admirable-start game, Admirable starts by playing $v_1$.  First suppose Impish does not play $v_6$ on \sout{their}  first move. On Admirable's second move, \sout{they play} $v_6$. For Admirable's third move, \sout{they avoid}  playing $v_4$. As Admirable played first and $|S|$ is even, this is always possible. \sout{This would result in a win by Admirable under Strategy 1 as Admirable avoided all bad sets while only playing on the set $\{v_1,v_4\}$ once.} Now suppose Impish plays on $v_6$ for \sout{their} first move. On Admirable's second move, \sout{they play} $v_2$. If Impish does not play $v_4$ on \sout{their}  second move (or \sout{ they pass}), Admirable would then play $v_4$, guaranteeing a win under Strategy 2. If Impish plays on $v_4$ in \sout{their} second move, Admirable would then play $v_5$, which would guarantee a win under Strategy 1 as $\{v_1,v_2,v_5\}$ is not a bad set under this strategy.}

\track{We first handle Admirable-start games on $S$. Admirable beings by playing $v_1$. First suppose Impish does not play $v_6$ on their first move. Then on Admirable's second move, he plays $v_6$. For Admirable third move, he avoids playing $v_4$. As Admirable played first and $|S|$ is even, this is always possible. The possible labelings for Admirable are 
\begin{align*}
    \{v_1,v_2,v_6\},  \;\;\; \{v_1,v_3,v_6\}, \;\;\; \{v_1,v_5,v_6\},
\end{align*} 
all of which are one of the possible twelve winning labelings. Now suppose Impish plays $v_6$ on his first move. On Admirable's second move, they play $v_2$. If Impish does not play $v_4$ on his second turn (or he passes), Admirable would then play $v_4$, Admirable would play $v_4$, obtaining the set $\{v_1,v_2,v_4\}$, a winning set. If Impish plays on $v_4$ on his second move, Admirable would then play $v_5$, which would guarantee the winning set $\{v_1,v_2,v_5\}$.
}

\sout{ For Impish-start games, by symmetry, we may assume Impish plays on $\{v_1,v_2,v_3\}$. First suppose Impish plays $v_1$. Admirable would then play $v_3$. If Impish does not play $v_4$ (or passes) for \sout{their} second move, Admirable plays $v_4$ for \sout{their} second move. Note then this guarantees a win for Admirable under Strategy 1. So suppose Impish plays $v_4$ on \sout{their} second move. From here Admirable would play $v_6$ on \sout{their}  second move. Irrespective of Admirable's third move, this guarantees a win under Strategy 2.}

\track{We now consider Impish-start games. By symmetry,  we may assume Impish plays on $\{v_1,v_2,v_3\}$. First suppose Impish plays $v_1$. Admirable would then play $v_3$. If Impish does not play $v_4$ (or passes) for his second move, Admirable plays $v_4$ for his second move. Regardless of Admirable's third move, he would obtain one of the following labelings,
\begin{align*}
    \{v_2,v_3,v_4\}, \;\;\; \{v_3,v_4,v_5\}, \;\;\; \{v_3,v_4,v_6\},
\end{align*}
all of which are winning. So suppose Impish plays $v_4$ on his second move. From here Admirable would play $v_6$ on his second move. Irrespective of Admirable's third move, he obtains one of the following winning labelings,
\begin{align*}
    \{v_2,v_3,v_6\}, \;\;\; \{v_3,v_5,v_6\}.
\end{align*}
}
\sout{
Now suppose Impish plays on $v_2$ for \sout{their} first move. Admirable would then follow by playing $v_5$. We break into cases depending on Impish's second move:
\begin{itemize}
    \item If Impish plays $v_6$ \sout{their} second move, Admirable would then play $v_4$. For Admirable's third move, if it is $v_3$, Admirable wins by Strategy 1. If Admirable third move is $v_1$, then Admirable's wins by Strategy 2. 
    \item If Impish passes or plays $v_4$ on \sout{their}  second move, Admirable would then play $v_6$ as \sout{their}  second move. Admirable's third move will be played on $\{v_1,v_3\}$, of which one will always be playable. Now if Admirable's third move was $v_1$, then Admirable wins by Strategy 1. If Admirable's third move was $v_3$, \sout{they win}  by Strategy 2.
    \item If Impish plays on $\{v_1,v_3\}$ on \sout{their}  second move, Admriable would then play $\{v_1,v_3\}$ as well. The final possible game states for Admirable are $\{v_1,v_4,v_5\}$, $\{v_1,v_5,v_6\}$, $\{v_3,v_4,v_5\}$, and $\{v_3,v_5,v_6\}$. Admirable wins under Strategy 1 if $\{v_1,v_5,v_6\}$ or $\{v_3,v_4,v_5\}$ are labeled, and wins with Strategy 2 otherwise.
\end{itemize}
}
\track{
Now suppose Impish plays $v_2$ on his first move. Admirable would then follow by playing $v_5$. Regardless of how Impish plays (or passes), Admirable can guarantee he labels exactly one vertex from $\{v_1,v_3\}$ and exactly one vertex from $\{v_4,v_6\}$. The possible labelings Admirable may obtain are
\begin{align*}
    \{v_1,v_4,v_5\} \;\;\; \{v_1,v_5,v_6\} \;\;\; \{v_3,v_4,v_5\} \;\;\; \{v_3,v_5,v_6\},
\end{align*}
all of which are winning.
}
\sout{ Now suppose Impish plays on $v_3$ for \sout{their}  first move, Admirable follows by playing $v_1$. If Impish does not play $v_5$ or passes on \sout{their}  second move, Admirable plays $v_5$. Suppose Admirable did not play $v_4$ on \sout{their}  third move, then \sout{they win} by Strategy 1. If Admirable plays $v_4$ on \sout{their}  third move, \sout{they played} $\{v_1,v_4,v_5\}$, i.e., \sout{they win} by Strategy 2. Thus we may suppose Impish plays $v_5$ on \sout{their}  second move. Then on Admirable's second move, \sout{they play}  $v_2$. If Admirable final game state is $\{v_1,v_2,v_6\}$, \sout{they win} by Strategy 1, if it is $\{v_1,v_2,v_4\}$, \sout{they win} by Strategy 2. {\hfill \rule{4pt}{7pt}}}

\track{Now suppose Impish plays on $v_3$ for his first move, Admirable would then follow by playing $v_1$. If Impish does not play $v_5$ or passes on their second move, Admirable plays $v_5$. The possible final labels for Admirable in this case would be
\begin{align*}
    \{v_1,v_2,v_5\} \;\;\; \{v_1,v_4,v_5\} \;\;\; \{v_1,v_5,v_6\}
\end{align*}
all of which are winning. If Impish plays $v_5$ on this second move, Admirable would then play $v_2$. The possible final game states for Admirable are
\begin{align*}
    \{v_1,v_2,v_6\} \;\;\; \{v_1,v_2,v_4\}
\end{align*}
which are both winning. {\hfill \rule{4pt}{7pt}}
}

\begin{case}\label{Case:7}
Branch $B$ is illustrated below with vertex $v$ in $T'$.
\end{case}

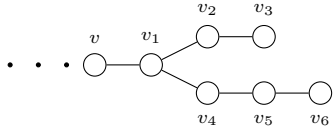
\begin{figure}[ht]
\begin{center}
\begin{tikzpicture}[scale=.75]
\tikzstyle{vert}=[circle,fill=black,inner sep=3pt]
\tikzstyle{overt}=[circle, draw, inner sep=3pt]

\node[overt, label=above:\tiny{$v$}] (v) at (1,1) {};
\node[overt, label=above:\tiny{$v_1$}] (v_1) at (2,1) {};
\node[overt, label=above:\tiny{$v_2$}] (v_2) at (3,1.5) {};
\node[overt, label=above:\tiny{$v_3$}] (v_3) at (4,1.5) {};
\node[overt, label=below:\tiny{$v_4$}] (v_4) at (3,.5) {};
\node[overt, label=below:\tiny{$v_5$}] (v_5) at (4,.5) {};
\node[overt, label=below:\tiny{$v_6$}] (v_6) at (5,.5) {};

\draw[color=black] 
(v)--(v_1)--(v_2)--(v_3) (v_1)--(v_4)--(v_5)--(v_6);

\filldraw[color=black]
    (.5,1) circle (1pt)
    (0,1) circle (1pt)
    (-.5,1) circle (1pt)
;

\end{tikzpicture}
%\caption{}
%\label{example}
\end{center}
\end{figure}

Let \track{the vertices of $B$ be labeled as pictured above and let} $S$ be the subgraph induced by $\{v_1,v_2,v_3,v_4,v_5,v_6\}$. Note that $S$ induces a subgraph isomorphic to $P_6$. By Lemma \ref{base}, the optimal discrepancy was shown to be $1$ for $P_6$, taking into the account the edge between $v$ and $v_1$ in $B$ produces a discrepancy of at most $2$. {\hfill \rule{4pt}{7pt}}

We now finish the proof by showing \track{ that} $T$ has a branch $B$ as described as in the prior seven cases. By Theorem \ref{ub}, we may assume $T$ has a vertex of degree at least three. Let $T'$ be the tree obtained from $T$ by suppressing all degree two vertices and let $T''$ be the tree obtained from $T'$ by deleting all leaves of $T'$. Let $u \in V(T'')$ be a leaf of $T''$. Note then $d_{T'}(u) = d_T(u) \ge 3$, as $u \in V(T')$ and $u$ was not a leaf of $T'$. \sout{Furthermore} \track{Thus,} $u$ must be adjacent to at least two leaves in $T'$, as \sout{$d_{T''}(u) = 1$} \track{$d_{T''}(u) \le 1$}. \track{Furthermore, if $|V(T'')| = 1$, we have that $d_{T''}(u) = 0$ and $u$ is adjacent to at least three leaves in $T'$.}

We now observe $u$ in $T$. Let $Q_1, \ldots, Q_k$ be the components of $T -u$ which are paths. \sout{ As \sout{$v$} \track{$u$} was adjacent to two leaves in $T'$, $k \ge 2$.}\sout{Furthermore, by our choice of $u$, $d_T(u) = k + 1$.}\track{If $|V(T'')| = 0$, then $u$ is adjacent to at least three leaves in $T'$, thus $d_T(u) = k \ge 3$. Otherwise, $|V(T'')| > 1$, $k \ge 2$, and by our choice of $u$, $d_T(u) = k + 1 \ge 3$.}

First suppose there exists $i$ such that $|Q_i| \ge 4$. Then there exists some $v \in V(Q_i) \cup \{u\}$ such that $v$ is the endpoint of a branch $B$ isomorphic to $P_5$. Thus by Case \ref{Case:1}, we may suppose $|Q_i| \le 3$ for all $i$. By Cases \ref{Case:2}, \ref{Case:3}, and \ref{Case:6}, for each $j \in \{1,2,3\}$, there can be at most one $i$ such that $|Q_i| = j$. By Case \ref{Case:5}, there cannot be two distinct $i,i'$ such that $|Q_i| = 1$ and $|Q_{i'}| = 3$. Note this implies $k = 2$, and without loss of generality, $|Q_1| = 2$ and either $|Q_2| = 1$ or $|Q_2| = 3$. Note this implies $d_T(u) = 3$ as well. Depending on the order of $Q_2$, either Case \ref{Case:4} or Case \ref{Case:7} is applicable, thus the claim follows.

\end{proof}

We do not believe that the upper bound in Theorem \ref{tree} is sharp. In fact, we posit that paths are the worst case for the \sout{Cordiality}\track{cordiality} game.

\begin{conj}
For any tree $T$ of order $n$, $c_g(T)\le c_g(P_n)$.
\end{conj}

\section{The Balance Game}

While studying cordiality game strategies, it becomes evident that Admirable cannot simply greedily label edges 0 or 1, as Impish can decide to ``play into'' this strategy and further imbalance the edge labelings. As a curiosity, one may wonder what happens if each player instead adopts a greedy objective of trying to maximize the number of edges with a particular label. To consider this proposition, we introduce a new game.

Let $G$ be a graph with vertices $V$ and edges $E$. The \emph{balance game} is played in a similar fashion as the cordiality game, but with a modified discrepancy $d$. We remove the absolute value from our previous definition of $d$ and define the (signed) discrepancy of the balance game to be $d = e_1 - e_0$. The objectives of Admirable and Impish are still to minimize and maximize $d$, respectively. However, these objectives can now be interpreted as Admirable and Impish attempting to maximize the number of 0s and 1s, respectively, in the labeling of $E$. We define the \emph{game balance number}, $b_g(G)$, to be the value of $d$ when both players play optimally. As in the cordiality game, this game too can be interpreted as an instance of a Maker-Breaker game. 
For $k \ge 0$, let ${\mathcal F}_k$ be the family of vertex sets $S$ such that $S$ and $V\setminus S$ is a balanced partition of $V$ and $ 2\phi(S) - |E(G)| \le k$. Again, Maker plays the role of Admirable, and Breaker plays the role of Impish. 

Notice that for any graph $G$, Admirable can choose to play according to the optimal cordiality game strategy on $G$. This will produce a balance game discrepancy of $\pm c_g(G)$, depending on whether 0s or 1s form a majority of the edge labels. Hence, we have $b_g(G) \leq c_g(G)$. Our proof method for paths translates to the balance game as the following theorem illustrates. 

\begin{thm}\label{lb} 
For any integer $n\ge 2$, $b_g(P_n) \geq 0$.
\end{thm}
\begin{proof}
    Let $P_n = v_1v_2 \cdots v_n$. Note the theorem follows if $n = 2$. We proceed by induction on $n \ge 3$. Let $P$ be the subpath of $P_n$ induced on $v_1, v_2, \dots, v_{n-2}$ and $P'$ be the $P_2$ subpath of $P$ induced by $v_{n-1}, v_n$. We present a strategy for Impish that, regardless of the moves Admirable makes, will produce the desired bound.

    If Admirable plays on $P$, Impish continues by playing the optimal game on $P$. If Admirable plays on $P'$, then Impish plays the remaining vertex on $P'$. If Admirable finishes a game on $P$ before moving to $P'$, Impish plays $v_n$. Regardless of the parity of $n$, this results in an Admirable-starts game on $P$ and the edge $v_{n-1}v_n$ being labeled 1. We apply the induction hypothesis on $P$ and consider that the edge between $P$ and $P'$, $v_{n-2}v_{n-1}$, has an unknown label, which produces the claimed lower bound of 0.
\end{proof}

We are not aware of any graph that violates this lower bound. We leave as an open question on whether there exists a graph $G$ with an accompanied Admirable strategy which produces a negative discrepancy.

\begin{openquestion}\label{blbc}
    For any graph $G$, does $b_g(G) \geq 0$?
\end{openquestion}

As an aside, the balance game has an interesting physical interpretation through the well-known Ising Model from statistical physics, e.g., see \cite[\track{Section 1.4.2}]{FV} \track{for the relevant definitions}. Two players assign spin states, \track{one assinging ``up", the other player ``down"}, to \track{the remaining unassigned} particles in alternating turns, competing over the 
energy of the final configuration at game completion. Open Question \ref{blbc}, reinterpreted, states: does the optimal strategy for the game keeps the Hamiltonian of the model nonnegative (with the assumption of an external magnetic field being absent)?

The complexity of Maker-Breaker games is well-studied, see for example \cite{B,RW,S}. As the Maker-Breaker instances of the cordiality game and balance game have a large amount of combinatorial structure on the associated set system, this leads to the following natural question. 

\begin{openquestion}
    Are the corresponding Maker-Breaker games to the cordiality and balance game PSPACE-complete?
\end{openquestion}

\section{Acknowledgments}
\track{We would like to thank the anonymous referee for their careful reading and their comments improving the paper. This work was done while third author was a graduate student at Georgia Tech and was supported by an NSF Graduate Research Fellowship under Grant No. DGE-1650044.}
\sout{Michael Wigal was supported by an NSF Graduate Research Fellowship under Grant No. DGE-1650044.}

 \bibliographystyle{plain}
 
 \end{document}